\documentclass{conm-p-l}

\usepackage{amssymb}

\usepackage[colorlinks=true, pdfstartview=FitV, linkcolor=blue, citecolor=blue, urlcolor=blue]{hyperref}

\usepackage{tikz}
\usepackage{hyperref}
\newcommand{\arxiv}[1]{\href{http://arxiv.org/abs/#1}{\tt arXiv:\nolinkurl{#1}}}

\newtheorem{theorem}{Theorem}[section]
\newtheorem{lemma}[theorem]{Lemma}

\theoremstyle{definition}

\newtheorem{example}[theorem]{Example}

\theoremstyle{remark}
\newtheorem{remark}[theorem]{Remark}

\numberwithin{equation}{section}

\newcommand{\bfi}{{\bf i}}
\newcommand{\g}{\mathfrak{g}}

\renewcommand{\phi}{\varphi}
\DeclareMathOperator{\wt}{wt}

\renewcommand{\tilde}{\widetilde}

\title[Elementary construction of Lusztig's canonical basis]{Elementary construction of Lusztig's canonical basis}

\author[Peter Tingley]{Peter Tingley}
\address{Department of Mathematics and Statistics, Loyola University, Chicago, IL, USA.}
\email{ptingley@luc.edu}
\thanks{Partially supported by NSF grant  DMS-1265555.}


\begin{document}

\begin{abstract}
In this largely expository article
we present an elementary construction of Lusztig's canonical basis in type ADE. The method, which is essentially Lusztig's original approach, is to use the braid group to reduce to rank two calculations. 
Some of the wonderful properties of the canonical basis are already visible: that it descends to a basis for every highest weight integrable representation, and that it is a crystal basis. 
\end{abstract}

\maketitle

\section{Introduction}
Fix a simple Lie algebra $\g$ over ${\Bbb C}$ and let $U^-_q(\g)$ be the lower triangular part of the corresponding quantized universal enveloping algebra. Lusztig's canonical basis $B$ is a basis for $U_q^-(\g)$, unique once the Chevalley generators are fixed, which has remarkable properties. Perhaps the three most important are:
\begin{enumerate}
\item \label{property:basis} For each finite dimensional irreducible representation $V_\lambda$, the non-zero elements in the image of $B$ in $V_\lambda = U^-_q(\g)/I_\lambda$ form a basis; equivalently, the intersection of $B$ with every ideal $I_\lambda$ is a basis for the ideal. 

\item \label{property:crystal} $B$ is a crystal basis in the sense of Kashiwara.

\item \label{property:positive} In symmetric type, the structure constants of $B$ with respect to multiplication are Laurent polynomials in $q$ with positive coefficients. 
\end{enumerate}
Much has been made of \eqref{property:positive}, and it helped give birth to a whole new field: categorification. While this is a wonderful fact, the association of canonical bases with categorification has, I believe, obscured the fact that Lusztig's original construction is quite elementary. Using only basic properties of the braid group action on $U_q(\g)$ and rank 2 calculations, one can establish the existence and uniqueness of a canonical basis, and show that it satisfies both \eqref{property:basis} and \eqref{property:crystal}. Property \eqref{property:positive} is mysterious with this approach, but perhaps that is to be expected, since it does not always hold is non-symmetric types (see \cite{Tsuchioka}), and the arguments here essentially work in all finite types. 

We present Lusztig's elementary construction, but with a few changes. Most notably, we have disentangled the construction from the quiver geometry Lusztig was studying at the same time. This has required modifying some arguments. In particular, our proof of Theorem \ref{thm:ut} differs from that presented by Lusztig. The results can be found in \cite{L90a,L90b, L90-2, L90c, L:1993}. 

Lusztig's canonical basis is the same as Kashiwara's global crystal basis \cite{Kash}, and Kashiwara's construction is also elementary, at least in the sense that it does not use categorification. However, Kashiwara's construction is quite different from that presented here, and considerably more difficult. It is based on a complicated induction known as the ``grand loop argument." 
Of course, Kashiwara's construction has a big advantage in that it works beyond finite type.

Leclerc \cite{Lec} has some related work, and in particular proves an analogue of our Theorem \ref{thm:ut} (see \cite[Lemma 37]{Lec}). Leclerc's argument is quite different from the one used here, but also avoids using quiver geometry. That work leads more naturally to the dual canonical basis.

This article is fairly self contained, the biggest exception being that we refer to Lusztig's book \cite{L:1993} for one elementary but long calculation in type $\mathfrak{sl}_3$. 
We restrict to the ADE case for simplicity. The construction is not much harder in other finite types, but requires some more notation. The rank two calculations are also considerably more difficult in types $B_2$ and $G_2$ (see \cite{BFZ}). Those cases can also be handled by using a folding argument to understand them in terms of types $A_{3}$ and $D_4$ respectively (see \cite{BZ,L11}). 

\subsection*{Acknowledgements}
We thank Steve Doty, George Lusztig, and Ben Salisbury for helpful comments. We also thank the anonymous referee for a very careful reading and for suggesting many improvements. 

\section{Notation}

Let $\g$ be a complex Lie algebra of type ADE, with a chosen Borel subalgebra $\mathfrak{b}$ and Cartan subalgebra $\mathfrak{h}$.  Let $U_q(\g)$ be its quantized universal enveloping algebra defined over ${\Bbb Q}(q)$ and let $E_i, F_i, K_i^{\pm 1}$ for $i \in I$ be the standard generators. Here $I$ indexes the nodes of the Dynkin diagram, so we can discuss elements being adjacent. Following \cite{Kash,Sai}, the defining relations are, for all $i \neq  j \in I$,
\begin{equation}
\begin{aligned}
& K_i K_i^{-1} =K_i^{-1} K_i= 1,  \  K_iK_j=K_jK_i, 
\, K_iE_iK_i^{-1}= q^2 E_i, \\ 
& K_iF_iK_i^{-1}= q^{-2} F_i, \
  E_i F_j - F_j E_i =0, \ \ E_i F_i - F_i E_i = \frac{K_i-K_i^{-1}}{q-q^{-1}}.  \\
& \text{If $i$ is adjacent to $j$: }  E_i^2 E_j + E_j E_i^2 = (q+q^{-1}) E_iE_j E_i, \ \  \\
& \mbox{} \hspace{3.2cm} F_i^2 F_j + F_j F_i^2 = (q+q^{-1}) F_iF_j F_i, \\
&  \hspace{3.2cm} K_i E_j K_i^{-1} =q^{-1} E_j, \ \ K_i F_j K_i^{-1} = q F_j. \\
& \text{Otherwise: } E_iE_j= E_j E_i,  \ F_iF_j= F_j F_i,  \ K_i E_j K_i^{-1} =E_j, \ K_i F_j K_i^{-1} = F_j. \\
\end{aligned}
\end{equation}
We use the standard triangular decomposition,
\begin{equation}
U_q(\g)= U_q^-(\g) \otimes U_q^0(\g) \otimes U_q^+(\g),
\end{equation}
where $U_q^-(\g)$ (resp. $U_q^0,$ or $U_q^{+}$)  is the subalgebra generated by the $F_i$ (resp. $K_i^{\pm 1}$ or $E_i$). We also use the triangular decomposition with the order of the factors reversed. 
Bar involution is the ${\Bbb Q}$-algebra involution defined on generators by
\begin{equation}
\bar{E}_i=E_i, \quad \bar F_i=F_i, \quad \bar K_i= K_i^{-1}, \quad \bar q= q^{-1}.
\end{equation}
 Let $\{ \alpha_i \}$ be the set of simple roots for $\g$. For a positive root $\beta$, define its height $\text{ht}(\beta)$ to be the sum of the coefficients when $\beta$ is written as a linear combinations of simple roots. 
 Let $(\cdot, \cdot )$ be the standard bilinear form on root space $\mathfrak{h}^*$.

\section{Braid group action and PBW bases}
The following can be found in \cite{L:1993}. Lusztig actually defines PBW bases for $U^+(\mathfrak{g})$, and uses a slightly different action of the braid group, but this causes no significant differences. For each $i \in I$ there is an algebra automorphism $T_i$ of $U_q(\g)$ (denoted $T''_{i,1}$ in \cite{L:1993}) given by
\begin{equation}
T_i(F_j) :=
\begin{cases}
F_j  \qquad \qquad \quad \;\; \hspace{0.05cm} i \text{ not adjacent to } j \\
F_j F_i - q F_i F_j \quad i \text{ adjacent to } j \\
-K_j^{-1} E_j \qquad  \;\;\;   i=j,
\end{cases}
\end{equation}
\begin{equation}
T_i(E_j) :=
\begin{cases}
E_j    \qquad \qquad \qquad \;\;\;  \ i \text{ not adjacent to } j \\
E_i E_j - q^{-1} E_j E_i \quad i \text{ adjacent to } j \\
-F_j K_j \qquad \qquad  \ \;\;\;  i=j, 
\end{cases}
\end{equation}
\begin{equation}
T_i(K_j) :=
\begin{cases}
K_j  \qquad \; i \text{ not adjacent to } j \\
K_i K_j  \quad i \text{ adjacent to } j \\
K_j^{-1} \quad \;\; i=j.
\end{cases}
\end{equation}
One can directly check that these respect the defining relations of $U_q(\g)$, and that they satisfy the braid relations 
(i.e. $T_i T_j T_i = T_j T_i T_j$ for $i$ and $j$ adjacent, and $T_i T_j=T_jT_i$ otherwise). 
Each $T_i$ performs the Weyl group reflection $s_i$ on weights, where $U_q(\g)$ is graded by $\wt(E_i)=-\wt(F_i)= \alpha_i$, $\wt(K_i)=0$. 

Fix a reduced expression $w_0 = s_{i_1} \cdots s_{i_N}$ for the longest element of the Weyl group. Let ${\bf i}$ denote the sequence $i_1, i_2, \ldots, i_N$. Define ``root vectors"
\begin{equation}
\begin{aligned} 
F_{{\bf i} ; \beta_1} & := F_{i_1} \\
F_{{\bf i} ; \beta_2} & := T_{i_1}F_{i_2}  \\
F_{{\bf i} ;  \beta_3} & := T_{i_1}T_{i_2} F_{i_3}    \\
& \;\;\; \vdots \;\;\; .
\end{aligned}
\end{equation}
The notation $\beta_k$ in the subscripts is because, for all $k$,  
\begin{equation} wt(F_{{\bf i}, \beta_k}) = -s_{i_1} \cdots s_{i_k-1} \alpha_{i_k}. \end{equation}
These are exactly the negative roots, and we index the root vectors by the corresponding positive roots $\beta_k$.
When it does not cause confusion we leave off the subscript ${\bf i}$. 

\begin{example}
If $\mathfrak{g}=\mathfrak{sl}_3$ and ${\bf i}$ corresponds to the reduced expression $s_1s_2s_1$ then $(\beta_1,\beta_2,\beta_3)= (\alpha_1, \alpha_1+\alpha_2, \alpha_2)$ and $(F_{\beta_1}, F_{\beta_2}, F_{\beta_3})=(F_1, F_2F_1-qF_1F_2, F_2)$. 
\end{example}

Let 
\begin{equation}
B_{\bf i} := \{ F_{{\bf i} ;  \beta_1}^{(a_1)} F_{{\bf i} ;  \beta_2}^{(a_2)} \cdots F_{{\bf i} ;  \beta_N}^{(a_N)} : a_1, \ldots, a_N \in \Bbb Z_{\geq 0} \}.
\end{equation}
Here $X^{(a)}$ is the $q$-divided power $X^a/([a][a-1] \cdots [2])$, and $[n]=q^{n-1} + q^{n-3} + \cdots + q^{-n+1}$. We call the collection of exponents ${\bf a}= (a_1, \ldots, a_N)$ for an element of $B_{\bf i}$ its Lusztig data, and denote the element by $F^{\bf a}_{\bf i}$.

\begin{remark}
One can define $B_{\bf i}$ for any reduced word, not just reduced expressions of $w_0$, and many of the results in this article still hold. In particular, this can be done outside of finite type, where there is no longest element. 
\end{remark}

\begin{lemma} \label{lem:bp} Fix a reduced expression ${\bf i}$. 
\begin{enumerate}

\item \label{bp1} If $i_k, i_{k+1}$ are not adjacent, then reversing their order gives another reduced expression ${\bf i'}$, and the root vectors are unchanged (although they are reordered, since $\beta'_k= \beta_{k+1}$, and  $\beta'_{k+1}= \beta_{k}$).  

\item \label{bp2}  If $i_k=i_{k+2}$ and is adjacent to $i_{k+1}$, then $\beta_k+\beta_{k+2}=\beta_{k+1}$ and 
$$F_{\beta_{k+1}}=  F_{\beta_{k+2}} F_{\beta_k} - q F_{\beta_k} F_{\beta_{k+2}}.$$ 
Furthermore, for the new reduced expression ${\bf i'}$ where $i_k i_{k+1} i_k$ is replaced with $i_{k+1} i_k i_{k+1}$, $F_{{\bf i'}, \beta} = F_{{\bf i}, \beta}$ for all $\beta \neq \beta_{k+1}$.  

\item \label{bp3} If $\beta_k=\alpha_i$ for some $k,i$, then $F_{{\bf i}; \beta_k}=F_i$.
In particular, $F_{\beta_N} = F_{{\sigma (i_N)}}$, where $\sigma$ is the Dynkin diagram automorphism given by $\alpha_{\sigma(i)}= -w_0 \alpha_i$. 

\end{enumerate}
\end{lemma}

\begin{proof}
Part \eqref{bp1} and \eqref{bp2} follow by applying $T_{i_{k-1}}^{-1} \cdots T_{i_1}^{-1}$ and then doing a rank two calculation. Part \eqref{bp3} is an immediate consequence of \eqref{bp2}, since $\alpha_i$ is not the sum of any two positive roots, and if $i_1=i$ then $F_{\alpha_i}=F_i$ by definition. 
\end{proof}

\begin{lemma} \label{lem:inm}
Each root vector $F_{{\bf i}; \beta_k}$ is in $U_q^-(\g)$.
\end{lemma}

\begin{proof}
Proceed by induction on the height of $\beta=\beta_k$, the case of a simple root being immediate from Lemma \ref{lem:bp} \eqref{bp3}.
So assume $\beta$ is not simple. Fix $i$ so that $( \alpha_i, \beta ) > 0$.  There are reduced expressions $\bfi'$ and $\bfi''$ with $i'_1=i$ and $i''_N=\sigma(i)$, so $\beta'_1=\beta''_N=\alpha_i$. By Matsumoto's Theorem \cite{Matsumoto} one can move from $\bfi$ to either $\bfi'$ or $\bfi''$ by sequences of braid moves, and one of these sequences must move $\alpha_i$ past $\beta$. At that step $F_{{\bf i}; \beta}$ changes. The first time $F_{{\bf i}; \beta}$ changes Lemma \ref{lem:bp} \eqref{bp2} allows us to conclude by induction that $F_{{\bf i}; \beta} \in U_q^-(\g)$. 
\end{proof}

\begin{lemma} \label{lem:rest-of-triangular}
If $j \geq k$, then $T_{i_j}^{-1} \cdots T_{i_1}^{-1} F_{{\bf i}; \beta_k} \in U^{\geq 0}_q(\g)$. 
\end{lemma}

\begin{proof}
\begin{equation}
T_{i_k}^{-1} \cdots T_{i_1}^{-1} F_{{\bf i}; \beta_k} = -K_{i_k}^{-1} E_{i_k},
\end{equation}
and ($i_{k+1} \cdots, i_N, \sigma(i_1), \cdots, \sigma(i_k))$ yields another reduced expression for $w_0$. The claim follows from Lemma \ref{lem:inm} (or more precisely an analogue with $F_i$ and $T_i$ replaced by $E_i$ and $T_i^{-1}$ respectively) since  the $T_i$ are algebra automorphisms and preserve $U^0_q(\g)$.
\end{proof}

\begin{theorem} \label{th:is-a-basis} For any ${\bf i}$, 
$B_{\bf i}$ is a ${\Bbb Q}(q)$-basis for $U_q^-(\g)$. 
\end{theorem}

\begin{proof}
The dimension of each weight space of $U_q^-(\g)$ is given by Kostant's partition function, so the size of the proposed basis is correct, and it suffices to show that these elements are linearly independent. Proceed by induction on $k$, showing that the set of such elements where $a_j=0$ for $j>k$ is linearly independent. The key is that 
\begin{equation}
T_{i_1}^{-1} F^{\bf a}= (- K_{i_1}^{-1} E_{i_1})^{(a_1)} \otimes F_{\bf i'}^{\bf a'} \in  U_q^{\geq 0} (\g) \otimes U_q^-(\g),
\end{equation}
where ${\bf i'}= (i_2, i_3, \ldots, i_N, \sigma(i_1))$ and ${\bf a'}= (a_2, a_3, \ldots, a_k, 0, \ldots, 0)$. 
The $F_{\bfi'}^{\bf a'}$ are linearly independent by induction, so the vectors $T_{i_1}^{-1} F_{\bf i}^{\bf a}$ are linearly independent by the triangular decomposition of $U_q(\g)$. The result follows since $T_{i_1}^{-1}$ is an algebra automorphism. 
\end{proof}

The following are referred to as convexity properties of PBW bases. 
\begin{lemma} \label{lem:convex} Fix ${\bf i}$ and $ 1 \leq j < k \leq N$. 
\begin{enumerate}

\item \label{con1} Write $F_{\beta_k} F_{\beta_j} = \sum_{\bf a} p_{\bf a} F^{\bf a}_{\bf i}$. If $p_{\bf a}\neq 0$ then the only factors that appear with non-zero exponent in $F_{\bf i}^{\bf a}$ are $F_{\beta_i}$ for $j \leq i \leq k$.

\item \label{con2} If $n \beta_\ell = a_j \beta_j+ \cdots + a_k \beta_k$ for $n, a_j,a_k>0$ and $a_{j+1}, \ldots, a_{k-1} \geq 0$, then $j < \ell < k$. 
\end{enumerate}
\end{lemma}

\begin{proof}
Since the $T_i$ are algebra automorphisms, Lemmas \ref{lem:inm} and \ref{lem:rest-of-triangular} give 
\begin{equation}
T_{i_{j-1}}^{-1} \cdots T_{i_2}^{-1} T_{i_1}^{-1} (F_{\beta_k} F_{\beta_j}) \in U_q^-(\g)
 \text{ and } 
T_{i_{k}}^{-1} \cdots T_{i_2}^{-1} T_{i_1}^{-1} (F_{\beta_k} F_{\beta_j}) \in U_q^{\geq 0}(\g).
\end{equation}
A linear combination of PBW basis elements can only satisfy these conditions if, in all of them, the exponents of $F_{\beta_i}$ are $0$ unless $j \leq i \leq k$. 
This establishes \eqref{con1}. 

For \eqref{con2}, 
Notice that 
$ s_{i_1} \cdots s_{i_{j-1}} (a_j \beta_j+ \cdots + a_k \beta_k)$ is in the positive span of the simple roots, and $s_{i_1} \cdots s_{i_{k}} (a_j \beta_j+ \cdots + a_k \beta_k)$ is in the negative span.
This can only happen for $n\beta_\ell$ if $j \leq \ell \leq k$. 
If $\ell=j$, then for weight reasons $n>a_j$. But then 
$(n-a_j) \beta_\ell = 0 \beta_j+a_{j+1} \beta_{j+1}+  \cdots + a_k \beta_k$ leads to a contradiction as above. A similar argument rules out $\ell=k$. 
\end{proof}

\begin{lemma} \label{lem:compspan}
Assume ${\bf i}, {\bf i'}$ are related by a single braid move. Fix a root $\beta$ such that $F_{{\bf i}, \beta}= F_{{\bf i'}, \beta}$. Then, for any $n$,
$$\operatorname{span} \{F^{\bf a}_{\bf i} \in B_{\bf i} :  F^{\bf a}_{\bf i}  \neq F_\beta^{(n)} \} = \operatorname{span} \{F^{\bf a}_{\bf i'} \in B_{\bf i'} :  F^{\bf a}_{\bf i'} \neq F_\beta^{(n)} \}.$$
\end{lemma}

\begin{proof}
For two term braid moves $B_{\bf i} = B_{\bf i'}$ and the result is trivial. So assume ${\bf i}, {\bf i'}$ are related by a three term braid move affecting positions $i,i+1,i+2$. $F_{\beta_{i+1}}$ changes with such a move, so $\beta \neq \beta_{i+1}$. If $\beta \neq \beta_{i}, \beta_{i+2}$, then the claim is also trivial. So, it suffices to consider the cases $\beta = \beta_{i}, \beta_{i+2}$, and by symmetry it is enough to consider $\beta = \beta_{i}$. We will check that any monomial in $B_{\bf i}$ that has a non-zero exponent of $F_\gamma$ for $\gamma \neq \beta$ is equal to a linear combination of monomials on $B_{\bf i'}$ that still all have a non-zero exponent for some root other than $\beta$. 

If $F^{\bf a}_{\bf i} \in B_{\bf i}$ has a non-zero exponent for some $j \neq i,i+1,i+2$, every monomial that appears in its $B_{\bf i'}$ expansion will have that same exponent. 
If a monomial is such that the only non-zero exponents are $a_i, a_{i+1}, a_{i+2}$, and one of $a_{i+1}, a_{i+2}$ is non-zero, then its weight does not equal $n \beta$, so $F_{\beta}^{(n)}$ cannot appear in its expansion in $B_{\bf i'}$. This exhausts the possibilities. 
\end{proof}

\section{Equality mod $q$ and piecewise linear bijections} \label{s:emq}

Fix a reduced expression ${\bf i}$ for $w_0$, and recall from Theorem \ref{th:is-a-basis} that $B_{\bf i}$ is a basis for $U_q^-(\mathfrak{g})$. Let
\begin{equation}
\mathcal{L}=  \operatorname{span}_{{\Bbb Z[q]}} B_{\bf i}.
\end{equation}
Part \eqref{eq:tpp1} of the following can be found in \cite[Proposition 41.1.4]{L:1993}, and \eqref{eq:tpp2} is part of  \cite[Proposition 42.1.5]{L:1993}.  For non-simply laced types see \cite{Sai}.

\begin{theorem} $\mbox{}$  \label{th51}
\begin{enumerate}
\item \label{eq:tpp1}
$\mathcal{L}$ is independent of ${\bf i}$. 
\item \label{eq:tpp2}
The basis $B_{\bf i} + q \mathcal{L}$ of $\mathcal{L}/q \mathcal{L}$ is independent of ${\bf  i}$.  
\end{enumerate}
\end{theorem}

\begin{proof}
Any two reduced expressions  are related by a sequence of braid moves, so it suffices to consider reduced expressions related by a single braid move. The case of a two-term braid move
 is trivial, so consider a three-term braid move involving $i_k=i, i_{k+1}=j, i_{k+2}=i$. It suffices to check that
\begin{equation} \label{eq:bms}
\operatorname{span}_{{\Bbb Z}[q]} \{ F_{{\bf i}; {\beta_k}}^{(a_k)} F_{{\bf i}; {\beta_{k+1}}}^{(a_{k+1})} F_{{\bf i}; {\beta_{k+2}}}^{(a_{k+2})} \}= 
\operatorname{span}_{{\Bbb Z}[q]} \{ F_{{\bf i'}; {\beta'_k}}^{(a_k)} F_{{\bf i'}; {\beta'_{k+1}}}^{(a_{k+1})} F_{{\bf i'}; {\beta'_{k+2}}}^{(a_{k+2})} \},
\end{equation}
and that these sets coincide modulo $q$. 
Applying $T_{i_{k-1}}^{-1} \cdots T_{i_{1}}^{-1}$ shows that this is equivalent to the statement in the $\mathfrak{sl}_3$ case. That is an explicit (although surprisingly difficult) calculation, which can be found in \cite[Chapter 42]{L:1993}.
\end{proof}

One often wants to understand how the Lusztig data changes when one applies a braid move. That is, given $F^{\bf a}_{\bf i} \in B_\bfi$, one would like to know which element of $B_{\bf i'}$ is equal to it mod $q$. This is described by Lusztig's piecewise linear bijections from \cite[Chapter 42]{L:1993}. For a two term braid move involving $i_k, i_{k+1}$, the exponents of all $F_\beta$ stay the same (although two of them change places, since the roots are reordered). For a three term braid move involving 
$i_k, i_{k+1}, i_{k+2}$, all the exponents stay the same except for $a_{k}, a_{k+1}, a_{k+2}$, and these change according to: 
\begin{equation}
\begin{aligned}
a'_k \;\; &= \max \{ a_{k+1}, a_{k+1}+ a_{k+2} - a_k \}, \\
a'_{k+1} &= \min \{ a_k, a_{k+2}\}, \\
a'_{k+2} &= \max \{ a_{k+1}, a_{k+1}+a_k - a_{k+2} \}.
\end{aligned}
\end{equation}

\section{Triangularity of bar involution and the canonical basis}

There are two natural lexicographical orders on Lusztig data: one where ${\bf a} < {\bf b}$ if $a_1 > b_1$ or $a_1=b_1$ and $(a_2, \ldots) < (b_2, \ldots)$, and the other where one starts by comparing $a_N$ and $b_N$.  
Consider the partial order $\prec$ where ${\bf a} \preceq {\bf b}$ if $\wt({\bf a})=\wt({\bf b})$ and ${\bf a}$ is less then ${\bf b}$ for both of these orders. It follows from Lemma \ref{lem:convex}\eqref{con2} that 
the minimal elements are those where $a_k \neq 0$ implies $\beta_k$ is a simple root. Data with a unique non-zero $a_k$ are maximal, and are in fact the unique maximal elements of weight $a_k \beta_k$. 

\begin{theorem} \label{thm:ut}
For every reduced expression ${\bf i}$ and every Lusztig data ${\bf a}$, 
$$\bar F_{\bf i}^{\bf a} = F_{\bf i}^{\bf a} + \sum_{{\bf a'}\prec {\bf a}} p^{\bf a}_{{\bf a'}}(q) F_{\bf i}^{{\bf a'}},$$
where the $p^{\bf a}_{{\bf a'}}(q)$ are Laurent polynomials in $q$. 
\end{theorem}

\begin{proof}
That the coefficients are Laurent polynomials follows from the form of bar and the braid group operators. 
The point is the unit triangularity.

If the claim is true for all $F^{(a_j)}_{\beta_j}$, then $\bar F^{\bf a}_{\bf i}$ would be equal to $F_{\bf i}^{\bf a}$ plus terms obtained by replacing some of the $F_{{\bf i};\beta}$ with lesser monomials. Lemma \ref{lem:convex} implies that, once this is rearranged, all terms that appear are $\prec F_{\bf i}^{\bf a}$. Hence the minimal counter-example would have to be of the form $F_{\beta}^{(n)} = F_{{\bf i}; \beta_j}^{(n)}$ for some ${\bf i}$, $j$ and $n$. 

Proceed by induction on the height $\text{ht}(\beta)$. By Lemma \ref{lem:bp}, $F_{\alpha_i}^{(n)}= F_{i}^{(n)}$ satisfies the condition (it is in fact bar-invariant), so assume $\text{ht}(\beta)>1$. 
Certainly 
\begin{equation} \label{eq:bfbj}
\bar F_{\beta}^{(n)} = p(q) F_{\beta}^{(n)} + \sum_{{\bf a'}\prec {\bf a}} p^{\bf a}_{{\bf a'}}(q) F^{{\bf a'}},
\end{equation}
since $F_{\beta}^{(n)}$ is the unique maximal element of its weight. It remains to show that $p(q)=1.$ 

First consider just $F_\beta$ (and please refer to Example \ref{ex:illa}).  
Do braid moves until $F_{\beta}$ changes (this is possible as discussed in the proof of Lemma \ref{lem:inm}). 
For the braid moves where $F_{\beta}$ does not change, by Lemma \ref{lem:compspan}, terms $\prec F_{\beta}$ get sent to linear combinations of terms that are still $\prec F_{\beta}$, so $p(q)$ does not change. Thus we may assume that a single braid move would change $F_{\beta}$. Then by Lemma \ref{lem:bp},
$F_{\beta_j}= F_{\beta_{j+1} } F_{\beta_{j-1}} - q F_{\beta_{j-1}} F_{\beta_{j+1}}$, so
\begin{equation} \label{eq:gat}
\begin{aligned}
& \bar F_{\beta_j} - F_{\beta_j} =  (\bar F_{\beta_{j+1} }- F_{\beta_{j+1} }) F_{\beta_{j-1}}  + \bar  F_{\beta_{j+1} }( \bar F_{\beta_{j-1} }- F_{\beta_{j-1} })  + \\
& \hspace{2.3cm} +q F_{\beta_{j-1}} F_{\beta_{j+1}}-q^{-1} \bar F_{\beta_{j-1}} \bar F_{\beta_{j+1}}.
\end{aligned}
\end{equation}
By induction the statement holds for $ F_{\beta_{j+1}}$, so
$\bar F_{\beta_{j+1} }- F_{\beta_{j+1} }$ is a sum of PBW monomials of weight $\beta_{j+1}$, all $\prec  F_{\beta_{j+1} }$. In particular, each has a left factor $F_{\beta_\ell}$ for some $\ell< j+1$, and for weight reasons we actually must have $\ell<j$. By Lemma \ref{lem:convex}\eqref{con1}, every term in the PBW expansion of $ (\bar F_{\beta_{j+1} }- F_{\beta_{j+1} }) F_{\beta_{j-1}} $ has a left factor  $F_{\beta_\ell}$ for $\ell <j$. Similar arguments show that every term in the PBW expansion of the remaining parts has either a left factor $F_{\beta_\ell}$ for $\ell <j$ or a right factor $F_{\beta_m}$ for $m >j$. Since $F_{\beta_j}$is the unique maximal PBW monomial of weight $\beta_j$ the statement holds.

Now consider $F_{\beta_j}^{(n)}$. We know $\bar F_{\beta_j}-F_{\beta_j}$ is a sum of terms $\prec F_{\beta_j}$, so
\begin{equation}
\bar F_{\beta_j}^{(n)} -F_{\beta_j}^{(n)}  =(F_{\beta_j} + (\bar F_{\beta_j}-F_{\beta_j}))^{(n)} -F_{\beta_j}^{(n)}  \end{equation} is a linear combination of terms of the form
\begin{equation} \label{eq:anew}
F_{\beta_j}^k (\text{ a PBW monomial $ M \prec F_{\beta_j}$ } ) R,
\end{equation}
where the precise form of $R$ is irrelevant. Each $M$ has a left factor $F_{\beta_\ell}$ for $\ell < j$. Applying Lemma \ref{lem:convex}\eqref{con1} repeatedly, every term in the PBW expansion of $F_{\beta_j}^k MR$ also has a left factor $F_{\beta_{\ell'}}$ for some $\ell' <j$, so is $\prec F_{\beta_j}^{(n)}.$ 
\end{proof}

\begin{example} \label{ex:illa} 
Consider $\mathfrak{sl}_4$ and the reduced expression $w_0=s_3s_1s_2s_1s_3s_2$. The corresponding order on positive roots is
\begin{equation}
\beta_1=\alpha_3, \  \beta_2=\alpha_1, \  \beta_3=\alpha_1+\alpha_2+\alpha_3, \  \beta_4= \alpha_2+\alpha_3, \  \beta_5=\alpha_1+\alpha_2, \  \beta_6=\alpha_2.
\end{equation}
Applying braid moves until the relevant $F_{\beta_k}$ changes, and using Lemma \ref{lem:bp}, gives $F_{ \beta_4}= F_2F_3-qF_3F_2$, and
$F_{\beta_3} = F_{ \beta_4} F_1 - qF_1F_{\beta_4}.$ 
Then
\begin{equation} \bar F_{ \beta_4} - F_{\beta_4}= (q-q^{-1})F_3F_2, \end{equation}
which is certainly $\prec F_{\beta_4}.$ We also have 
\begin{equation} \bar F_{ \beta_3} - F_{\beta_3} = \bar F_{ \beta_4} F_1 -F_{ \beta_4} F_1+ q F_1 F_{ \beta_4} -q^{-1} F_1\bar F_{ \beta_4}.\end{equation}
This is simpler than \eqref{eq:gat} because $F_1$ is bar invariant. 
Inductively, the right side is
\begin{equation} (\text{terms $\prec F_{ \beta_4}$}) F_1 + F_1 (\text{ something }).\end{equation}
The terms $\prec F_{\beta_4}$ all have factors $F_{\beta_k}$ for $k <3$, a property which is preserved under right multiplication by Lemma \ref{lem:convex}, so all terms that appear when one rearranges are $\prec F_{\beta_3}$. Here the only term $\prec F_{\beta_4}$ is $F_3F_2$ so this can also be verified directly. 
\end{example}

\begin{minipage}{4.7in}
\begin{theorem} \label{th:cbe}
There is a unique basis $B$ of $U_q^-(\g)$ such that
\begin{enumerate}
\item $B$ is contained in $\mathcal{L}$, $B+q\mathcal{L}$ is a basis for $\mathcal{L}/q\mathcal{L}$, and this agrees with $B_{\bf i} + q\mathcal{L}$ for some (equivalently any by Theorem \ref{th51}) ${\bf i}$. 
\item $B$ is bar invariant. 
\end{enumerate}
Furthermore, the change of basis from any $B_{\bf i}$ to $B$ is unit-triangular. 
\end{theorem}
\end{minipage}

\begin{proof}
This proof can be found in \cite[\S 5.1]{Lec}  and \cite[Lemma 0.27]{DDPW} in slightly different settings.
Fix ${\bf i}$ and proceed by induction on the partial order $\prec$, proving that there is such a basis for 
$V_{\bf a}= \operatorname{span} \{F^{\bf a'} \}_{{\bf a'} \preceq {\bf a}} $.  The case when ${\bf a}$ is minimal holds since Theorem \ref{thm:ut} shows that $F^{\bf a}$ itself is bar-invariant. 

So, fix a non-minimal ${\bf a}$. By Theorem \ref{thm:ut},
\begin{equation}
\bar F^{\bf a}= F^{\bf a} + \sum_{{\bf a'} \prec {\bf a}} p_{\bf a'}^{\bf a}(q)  b^{\bf a'}
\end{equation}
for various Laurent polynomials $p_{\bf a'}^{\bf a}(q)$,
where the $b^{\bf a'}$ are the inductively found elements of $B$.
But $\bar{\bar F}^{\bf a}=F^{\bf a}$, which implies that
each $p_{\bf a'}^{\bf a}(q)$ is of the form 
\begin{equation}
p_{\bf a'}^{\bf a}(q)  = q {f}_{\bf a'}^{\bf a}(q)- q^{-1} f_{\bf a'}^{\bf a}(q^{-1}),
\end{equation}
where each $f_{\bf a'}^{\bf a}(q) $ is a polynomial. 
Set 
\begin{equation} \label{eq:ba} 
b^{\bf a} = F^{\bf a} + \sum_{{\bf a'} \prec {\bf a}} q f_{\bf a'}^{\bf a}(q)  b^{\bf a'}.
\end{equation}
Replacing $F^{\bf a}$ with $b^{\bf a}$ does not change $\mathcal{L}$ and $b^{\bf a}= F^{\bf a}$ mod $q \mathcal{L}$. Then
\begin{equation}
\begin{aligned}
\bar b^{\bf a} & =  F^{\bf a} + \sum_{{\bf a'} \prec{\bf a}} (q {f}_{\bf a'}^{\bf a}(q)- q^{-1} f_{\bf a'}^{\bf a}(q^{-1})) b^{\bf a'} + \sum_{{\bf a'} \prec {\bf a}} q^{-1} f_{\bf a'}^{\bf a}(q^{-1})  b^{\bf a'} \\
& = F^{\bf a} + \sum_{{\bf a'} \prec {\bf a}} q f_{\bf a'}^{\bf a}(q)  b^{\bf a'} = b^{\bf a},
\end{aligned}
\end{equation}
so we have found the desired element.

Uniqueness is clear, since as the induction proceeds there is no choice. 
\end{proof}

\begin{remark}
The basis $B$ from Theorem \ref{th:cbe} is Lusztig's canonical basis (see \cite[Theorem 3.2]{L90b}). As in the above proof, it can be indexed as $B = \{ b^{\bf a} \}$ where the $\bf a$ are Lusztig data with respect to a fixed reduced expression of $w_0$. However, as in \S\ref{s:emq}, the indexing changes depending on the reduced expression. 
\end{remark}

\section{Properties of the canonical basis}

\subsection{Descent to modules}

\begin{theorem}
Fix a dominant integral weight $\lambda$ and write $V_\lambda= U^-_q(\g)/I_\lambda$. Then $B \cap I_\lambda$ spans $I_\lambda$. Equivalently, $\{ b+I_\lambda : b \in B, b \not\in I_\lambda \}$ is a basis for $V_\lambda$.  
\end{theorem}

\begin{proof}

Write $\lambda$ using fundamental weights, $\lambda = \sum c_i \omega_i$.
It is well known that 
\begin{equation} \label{eq:BGG}
I_\lambda = \sum_{i \in I} U^-_q(\g) F_i^{c_i+1}.
\end{equation}
Thus it suffices to show that
$B \cap U^-_q(\g) F_i^{n}$ spans $U^-_q(\g) F_i^{n}$ for all $n$. 

Fix a reduced expression ${\bf i}$ with $i_N=\sigma(i)$, so that $F_{\beta_N}=F_i$. Then it is clear that $B_{\bf i} \cap U^-_q(\g) F_i^{n}$ spans $U^-_q(\g) F_i^{n}$. The change of basis from $B_{\bf i}$ to $B$ is upper triangular, so the canonical basis elements corresponding to elements in $B_{\bf i} \cap U^-_q(\g)  F_i^n $
are all still in $U^-_q(\g) F_i^{n}$, giving a spanning set. 
\end{proof}

\subsection{Crystal combinatorics}  \label{ss:Lcrystal}
In a sense we already have a combinatorial object that could be called a crystal. With that point of view the underlying set is the basis $B + q \mathcal{L}$ of $\mathcal{L}/q \mathcal{L}$. 
To perform a crystal operator $f_i$, choose a reduced expression ${\bf i}$ where $i_1=i$. On $B_{\bf i}$,
define 
\begin{equation} \label{eq:LCEG}
 f_i ( F_i^{(a_1)} F_{\beta_2}^{(a_2)} \cdots F_{\beta_N}^{(a_N)} )=  F_i^{(a_1+1)} F_{\beta_2}^{(a_2)} \cdots F_{\beta_N}^{(a_N)}.
 \end{equation}
This descends to an operation on $ B_{\bf i} +q \mathcal{L} =B+ q \mathcal{L}$. 
One must use different reduced expressions to define each $f_i$, and the full structure is somewhat complex. 

Since $B$ itself can be hard to work with, we often choose a reduced expression ${\bf i}$, and think of the crystal operators as acting on $B_{\bf i} + q \mathcal{L}$ (which is of course equivalent). With this point of view, the crystal operator $f_i$ acts as follows (see \S\ref{s:ex} for an example).
\begin{itemize}

\item Perform a series of braid moves to get a new reduced expression ${\bf i'}$ with $i'_i=i$, and use the piecewise linear functions to find the $F^{\bf a'}_{\bf i'} \in B_{\bf i'}$ which is equal to $F^{\bf a}_{\bf i}$ mod $q$.

\item Add 1 to $a'_1$. 

\item Perform a series of braid moves to get ${\bf i'}$ back to ${\bf i}$ and use the piecewise linear bijections to find the corresponding $F^{\bf \bar a}_{\bf i} \in B_{\bf i}$.
Then $f_i(F^{\bf a}_{\bf i}) = F^{\bf \bar a}_{\bf i}$. 

\end{itemize}

We now show that the structure defined above matches Kashiwara's crystal $B(\infty)$ from \cite{Kash}. This has previously been observed by Lusztig \cite{L90-2} (see also \cite{GL93,L11}) and by Saito \cite{Sai}. We give a somewhat different proof. 

We first review Kashiwara's construction of $B(\infty)$, roughly following \cite[\S 3]{Kash}.
For each $i \in I$, elementary calculations show that, for any $X \in U_q^-(\mathfrak{g})$,
\begin{equation} \label{eq:XAB}
E_i X= P K_i^{-1} + Q  K_i + X  E_i 
\end{equation}
for some $P,Q \in U_q^-(\mathfrak{g})$.  
Define $e'_i: U_q^-(\g) \rightarrow U_q^-(\g)$ by 
$e'_i(X)= P$. 
As a vector space, 
\begin{equation} \label{eq:eif}
U_q^-(\mathfrak{g}) \cong {\Bbb Q}(q)[F_i] \otimes \text{ker}(e'_i),
\end{equation}
where the isomorphism is multiplication. Define operators $\tilde F_i$ (the Kashiwara operators) by, for all $Y \in \text{ker}(e'_i)$ and $n \geq 0$, 
\begin{equation}
\tilde F_i (F_i^{(n)} Y)= F_i^{(n+1)} Y.
\end{equation}

Let ${\Bbb Q}[q]_0$ be the ring of rational functions which are regular at $q=0$, and
let $\mathcal{L}(\infty)$ to be the ${\Bbb Q}[q]_0$ lattice generated by all sequences of $\tilde F_i$ acting on $1 \in U_q^-(\g)$. There is a unique basis $B(\infty)$ for $\mathcal{L}(\infty)/q\mathcal{L}(\infty)$ such that the residues of all the $\tilde F_i$ act by partial permutations. This basis, along with the residues of the $\tilde F_i$, is $B(\infty)$.

\begin{theorem} \label{th:can-crystal}
Let $B$ be the canonical basis from Theorem \ref{th:cbe}. Then  
$\mathcal{L}(\infty)= \operatorname{span}_{{\Bbb Q}[q]_0} B$, and $B(\infty)=B + q \mathcal{L}(\infty)$. Furthermore, the crystal operators $\tilde F_i$ mod $q$ coincide with the operators described at the beginning of \S\ref{ss:Lcrystal}. 
\end{theorem}

Before proving Theorem \ref{th:can-crystal} we need some preliminary Lemmas. 

\begin{lemma} \label{lem:seqmoves}
Fix $i \in I$, a reduced expression ${\bf i}$, and a positive root $\beta$ with $( \beta, \alpha_i )\leq 0$. Then there is a sequence of braid moves, none of which affect the relative positions of $\alpha_i$ and $\beta$ in the corresponding order on roots, with the last move being a three term braid move with $\beta$ the middle root (so that $F_\beta$ changes).  
\end{lemma}

\begin{proof}
Fix $j,k$ so that  $\beta_j=\alpha_i$ and $\beta_k=\beta$. Without loss of generality $j<k$.
The prefix $w=s_{i_1} \cdots s_{i_{j}}$ satisfies $w^{-1} \alpha_i = - \alpha_j$, which is a negative root, so $w$ has a reduced expression of the form $ s_i \cdots$. One can perform a sequence of braid moves relating these two reduced expressions which do not change the position of $\beta$. Thus we may assume $i_1=i$. Since
$( \beta, \alpha_i ) \leq 0$ and $( \beta, \rho ) > 0$, we must have
$( \beta, \alpha_\ell ) > 0$ for some other $\ell$. 

 If $( \alpha_i, \alpha_\ell ) =0$, then there are reduced expressions for $w_0$ of the form 
\begin{equation} s_i s_\ell \cdots \quad \text{ and } \quad s_i \cdots s_{\sigma (\ell)}, \end{equation}
and both can be reached from ${\bf i}$ by performing braid moves that do not change the position of $\alpha_i$. Certainly the relative positions of $\beta$ and $\alpha_\ell$ are different in these two expressions, so one of these sequences moves $\beta$ past $\alpha_\ell$. Since $( \beta, \alpha_\ell ) > 0$, at that step $\beta$ is the middle root for a 3 term braid move.

 If $(\alpha_i, \alpha_\ell ) =-1$, then there are reduced expressions for $w_0$ of the form 
\begin{equation}
s_i s_\ell s_i \cdots \quad \text{ and } \quad s_i \cdots s_{\sigma (\ell)},
\end{equation}
and the same argument works. 
\end{proof}

\begin{lemma} \label{lem:acts-on-vector}
Fix a reduced expression ${\bf i}$, and let $j$ be such that $\beta_j = \alpha_i$ is a simple root. For all $k>j$, $$E_i F_{\beta_k} - F_{\beta_k} E_i \in U_q^-(\mathfrak{g}) K_i.$$ 
\end{lemma}

\begin{proof}
Proceed by induction on the height $\text{ht}(\beta_k)$, the case where $\beta_k$ is a simple root $\alpha_\ell \neq \alpha_i$ being trivial since $E_i F_{\beta_k} - F_{\beta_k} E_i=0$ by Serre's relations. 

So, assume $\text{ht}(\beta_k) \geq 2$. 
If $( \beta_k, \alpha_i ) \leq 0$, then by Lemma \ref{lem:seqmoves} we can do a sequence of braid moves that don't change the relative positions of $\alpha_i$ and $\beta_k$ and so that the last is a three term move with $\beta$ in the middle. 
At that step, by Lemma \ref{lem:bp},
\begin{equation}
F_{\beta_k} = F_{\beta_{k+1}} F_{\beta_{k-1}} - q F_{\beta_{k-1}} F_{\beta_{k+1}},
\end{equation}  
where $\text{ht}(\beta_{k-1}), \text{ht}(\beta_{k+1}) < \text{ht}(\beta_k)$. The claim holds for $ F_{\beta_{k-1}}$ and $F_{\beta_{k+1}}$ by induction, and so it easily follows for $F_{\beta_k}$. 

If $( \beta_k, \alpha_i ) >0$, perform any sequence of braid moves until $\beta_k$ is the middle term of a three term move. If 
 $\alpha_i$ has not moved past $\beta_k$ the result follows as in the previous paragraph. Otherwise at the step 
where $\alpha_i$ moves past $\beta_k$, we see that $\beta_k$ is the middle term of a three term move affecting the roots $\alpha_i, \beta_k, \beta_k-\alpha_i$, so, again using Lemma \ref{lem:bp},
\begin{equation}
\begin{aligned}
E_i F_{\beta_k}  \hspace{-0.1cm} -  F_{\beta_k} E_i& = E_i \left( F_{\beta_{k}-\alpha_i}F_i  - q F_i F_{\beta_{k}-\alpha_i} \right) - \left( F_{\beta_{k}-\alpha_i}F_i  - q F_i F_{\beta_{k}-\alpha_i} \right) E_i\\
&= F_{\beta_{k}-\alpha_i} \frac{K_i-K_i^{-1}}{q-q^{-1}}  -q  \frac{K_i-K_i^{-1}}{q-q^{-1}} F_{\beta_{k}-\alpha_i} + \text{terms in $U_q^- K_i$.}
\end{aligned}
\end{equation}
The fact that the other terms are in $U_q^- K_i$ uses induction. The claim follows since $K_i^{-1} F_{\beta_{k}-\alpha_i} K_i = q^{-1} F_{\beta_{k}-\alpha_i}$. 
\end{proof}

\begin{lemma} \label{lem:goodi} 
Fix $i$ and $\bfi$ such that $i_1=i$. Then 
$$\ker e_i'= \operatorname{span} \{ F_{\beta_2}^{(a_2)} \cdots F_{\beta_{N}}^{(a_{N})}  \};$$
that is, the span of PBW basis elements where the exponent of $F_i$ is 0. In particular, $\tilde F_i$ acts on $B_\bfi$ by simply increasing the exponent of $F_i$ by 1. 
\end{lemma}

\begin{proof}
Certainly 
$E_i F_{\beta_2}^{(a_2)} \cdots F_{\beta_{N}}^{(a_{N})}$
is equal to $ F_{\beta_2}^{(a_2)} \cdots F_{\beta_{N}}^{(a_{N})} E_i$
plus a sum of terms each of which is a PBW monomial but with one root vector $F_\beta$ replaced by $E_i F_\beta-F_\beta E_i$. By Lemma \ref{lem:acts-on-vector} each of these is in $U_q^-(\mathfrak{g}) K_i$. Therefore, by definition, each 
$F_{\beta_2}^{(a_2)} \cdots F_{\beta_{N}}^{(a_{N})} $ is in $\ker e_i'$. It follows from \eqref{eq:eif} that the span of these vectors has the correct graded dimension, so is the whole kernel. 
\end{proof}

\begin{proof}[Proof of Theorem \ref{th:can-crystal}]
Fix $i$, and choose $\bfi$ such that $i_1=i$. By Lemma \ref{lem:goodi}, $\tilde F_i$ acts by partial permutations on the basis $B_\bfi$. By a simple inductive argument, this implies that $ \operatorname{span}_{{\Bbb Q}[q]_0} B_{\bf i}=  \operatorname{span}_{{\Bbb Q}[q]_0} B$ is the lattice generated by all sequences of $\tilde F_i$ acting on $1 \in U_q^-(\g)$. That is, it is $\mathcal{L}(\infty)$. It also shows that $\tilde F_i$ acts on $B_\bfi$ as in \eqref{eq:LCEG}, and hence agrees with the crystal operators described at the beginning of this section. 
\end{proof}

\section{Example:  Crystal operators from piecewise linear bijections} \label{s:ex}
As in \S\ref{ss:Lcrystal}, one can develop crystal theory entirely within Lusztig's setup, where the 
underlying set is $B_{\bf i}+q \mathcal{L}$ for a fixed ${\bf i}$. To illustrate, take $\mathfrak{g}=\mathfrak{sl}_4$ and the reduced expression $w_0= s_1s_2s_3s_1s_2s_1$. The corresponding order on positive roots is
\begin{equation}
\alpha_1, \quad  \alpha_1+\alpha_2, \quad  \alpha_1+\alpha_2+\alpha_3, \quad  \alpha_2, \quad  \alpha_2+\alpha_3, \quad  \alpha_3.
\end{equation}
Consider 
\begin{equation}
x= F_1^{(2)} F_{12}^{(3)} F_{123}^{(1)} F_2^{(2)} F_{23}^{(4)} F_3^{(2)}\in B_{\bf i}.
\end{equation}
Here we use e.g. $F_{23}$ to mean $F_{\alpha_2+\alpha_3}$. Applying $f_1$ is easy: just increase the exponent of $F_1$ to $(3)$.  Figure \ref{fig:Lf3} shows the calculation of $f_3(x)$.

\begin{figure}
\begin{tikzpicture}[scale=1.44]

\draw node at (-0.8,4) {$x=$};

\draw node at (0,4) {$F_{1}^{(2)}$};
\draw node at (1,4) {$F_{12}^{(3)}$};
\draw node at (2,4) {$F_{123}^{(1)}$};
\draw node at (3,4) {$F_2^{(3)}$};
\draw node at (4,4) {$F_{23}^{(3)}$};
\draw node at (5,4) {$F_3^{(2)}$};
\draw node at (5.5,4) {$\mbox{}$};

\draw node at (0,3.5) {$F_{1}^{(2)}$};
\draw node at (1,3.5) {${F_{12}}^{(3)}$};
\draw node at (2,3.5) {$F_{123}^{(1)}$};
{\color{red} \draw node at (3,3.5) {$F_3^{(3)}$};
\draw node at (4,3.5) {$F_{32}^{(2)}$};
\draw node at (5,3.5) {$F_2^{(4)}$};}

\draw node at (0,3) {$F_{1}^{(2)}$};
{\color{red} \draw node at (1,3) {${F_{3}}^{(1)}$};
\draw node at (2,3) {$F_{312}^{(3)}$};
\draw node at (3,3) {$F_{12}^{(1)}$};}
\draw node at (4,3) {$F_{32}^{(2)}$};
\draw node at (5,3) {$F_2^{(4)}$};

{\color{red} \draw node at (0,2.5) {$F_{3}^{(1)}$};
\draw node at (1,2.5) {${F_{1}}^{(2)}$};}
\draw node at (2,2.5) {$F_{312}^{(3)}$};
\draw node at (3,2.5) {$F_{12}^{(1)}$};
\draw node at (4,2.5) {$F_{32}^{(2)}$};
\draw node at (5,2.5) {$F_2^{(4)}$};

\draw node at (0,2) {{\color{red} $F_{3}^{(2)}$}};
\draw node at (1,2) {${F_{1}}^{(2)}$};
\draw node at (2,2) {$F_{312}^{(3)}$};
\draw node at (3,2) {$F_{12}^{(1)}$};
\draw node at (4,2) {$F_{32}^{(2)}$};
\draw node at (5,2) {$F_2^{(4)}$};

\draw node at (0,1.5) {$\vdots$};
\draw node at (1,1.5) {$\vdots$};
\draw node at (2,1.5) {$\vdots$};
\draw node at (3,1.5) {$\vdots$};
\draw node at (4,1.5) {$\vdots$};
\draw node at (5,1.5) {$\vdots$};

\draw node at (-0.8,1) {$f_3(x)=$};

\draw node at (0,1) {$F_{1}^{(2)}$};
\draw node at (1,1) {${F_{12}}^{(3)}$};
\draw node at (2,1) {$F_{123}^{(1)}$};
\draw node at (3,1) {{\color{blue} $F_{2}^{(2)}$}};
\draw node at (4,1) {{\color{blue} $F_{23}^{(4)}$}};
\draw node at (5,1) {$F_3^{(2)}$};

\end{tikzpicture}

\caption{Calculation of $f_3(x)$. 
Lines 2-4 each show the PBW monomial obtained after applying a braid move and the corresponding piecewise linear bijection. The factors colored red have changed. The notation e.g. $F_{312}$ just means $F_{\alpha_1+\alpha_2+\alpha_3}$, but we distinguish between it and $F_{123}$ since root vectors depend on the reduced expression, and they are in fact different. We skip the steps of applying braid moves and piecewise linear bijections to get back to the original reduced expression. 
\label{fig:Lf3}}
\end{figure}

For this reduced expression things work out nicely: at most two exponents change when one applies an $f_i$, and, as discussed in \cite{CT}, there is a straightforward relationship with the well known crystal structure on semi-standard Young tableaux. 
 There are reduced expressions with similar behavior in types $D_n, E_6,$ and $E_7$ (see \cite{SST}).

In general the relationship with standard combinatorial models is more complicated. For instance, for the reduced expression 
$w_0=s_1s_3s_2s_1s_3s_2$,
\begin{equation}
f_2(F_1^{(2)}F_3^{(3)}F_{123}^{(3)}F_{23}^{(2)}F_{12}^{(3)}F_2^{(2)})=
F_1^{(2)}F_3^{(3)}F_{123}^{(2)}F_{23}^{(3)}F_{12}^{(4)}F_2^{(2)}.
\end{equation}
Notice that $3$ exponents have changed.

\end{document}